\documentclass[11pt]{article}
\usepackage{amsxtra}
\usepackage{amssymb, amsmath}
\usepackage{amsthm}
\usepackage{hyperref}
\usepackage{multicol}
\usepackage[dvips]{graphicx}

\newtheorem{lemma}{Lemma}[section]
\newtheorem{theorem}{Theorem}[section]
\newtheorem{remark}{Remark}[section]

\newtheorem{proposition}{Proposition}[section]

\usepackage[T1]{fontenc} 
\usepackage{palatino} 

\begin{document}
\begin{center}
\textbf{\LARGE{Inequalities having Seven Means \\ and Proportionality Relations}}
\end{center}

\smallskip
\begin{center}
\textbf{\large{Inder J. Taneja}}\\
Departamento de Matem\'{a}tica\\
Universidade Federal de Santa Catarina\\
88.040-900 Florian\'{o}polis, SC, Brazil.\\
\textit{e-mail: ijtaneja@gmail.com\\
http://www.mtm.ufsc.br/$\sim $taneja}
\end{center}

\begin{abstract}
In 2003, Eve \cite{eve}, studied seven means from geometrical point of view. These means are \textit{Harmonic, Geometric, Arithmetic, Heronian, Contra-harmonic, Root-mean square and Centroidal mean}. Some of these means are particular cases of Gini's \cite{gin} mean of order $r$ and $s$. In this paper we have established some proportionality relations having these means. Some inequalities among some of differences arising due to seven means inequalities are also established.
\end{abstract}



\section{Seven Geometrical Means}

Let $a,\,b>0$ be two positive numbers. In 2003, Eves \cite{eve} studied the geometrical
interpretation of the following seven menas:
\begin{enumerate}
\item Arithmetic mean: $A(a,b)={\left( {a+b} \right)} \mathord{\left/ {\vphantom {{\left( {a+b} \right)} 2}} \right. \kern-\nulldelimiterspace} 2$;
\item Geometric mean: $G(a,b)=\sqrt {ab} $;
\item Harmonic mean: $H(a,b)={2ab} \mathord{\left/ {\vphantom {{2ab} {\left( {a+b} \right)}}} \right. \kern-\nulldelimiterspace} {\left( {a+b} \right)}$;
\item Heronian mean: $N(a,b)={\left( {a+\sqrt {ab} +b} \right)} \mathord{\left/ {\vphantom {{\left( {a+\sqrt {ab} +b} \right)} 3}} \right. \kern-\nulldelimiterspace} 3$;
\item Contra-harmonic mean: $C(a,b)={\left( {a^{2}+b^{2}} \right)} \mathord{\left/ {\vphantom {{\left( {a^{2}+b^{2}} \right)} {\left( {a+b} \right)}}} \right. \kern-\nulldelimiterspace} {\left( {a+b} \right)}$
\item Root-mean-square: $S(a,b)=\sqrt {{\left( {a^{2}+b^{2}} \right)} \mathord{\left/ {\vphantom {{\left( {a^{2}+b^{2}} \right)} 2}} \right. \kern-\nulldelimiterspace} 2} $
\item Centroidal mean: $R(a,b)={2\left( {a^{2}+ab+b^{2}} \right)} \mathord{\left/ {\vphantom {{2\left( {a^{2}+ab+b^{2}} \right)} {3\left( {a+b} \right)}}} \right. \kern-\nulldelimiterspace} {3\left( {a+b} \right)}$
\end{enumerate}

Except 4 and 7 the above means are particular cases of well-known Gini
\cite{gin} mean of order $r$ and $s$ is given by
\begin{equation}
\label{eq1}
E_{r,s} (a,b)=\begin{cases}
 {\left( {\frac{a^{r}+b^{r}}{a^{s}+b^{s}}} \right)^{\frac{1}{r-s}}} & {r\ne
s} \\
 {\sqrt {ab} } & {r=s=0} \\
 {\exp \left( {\frac{a^{r}\ln a+b^{r}\ln b}{a^{r}+b^{r}}} \right)} & {r=s\ne
0} \\
\end{cases}.
\end{equation}

In particular, we have $E_{-1,0} =H$, $E_{-1/2,1/2} =G$, $E_{0,1} =A$,
$E_{0,2} =S$ and $E_{1,2} =R$. Since $E_{r,s} =E_{s,r} $, the Gini-mean
$E_{r,s} (a,b)$ is an increasing function in $r$ or $s$ \cite{czp}. In
view of this we have $H\le G\le A\le S\le C$. Moreover we can easily verify
the following inequality having the above seven means:
\begin{equation}
\label{eq2}
H\le G\le N\le A\le R\le S\le C.
\end{equation}
We can write, $M(a,b)=b\,f_{M} (a/b)$, where $M$stands for any of the above
seven means, then we have
\begin{equation}
\label{eq3}
f_{H} (x)\le f_{G} (x)
\le f_{N} (x))\le f_{A} (x)\le f_{R} (x)\le f_{S} (x)\le f_{C} (x),
\end{equation}

\noindent where $f_{H} (x)=2x/(x+1)$, $f_{G} (x)=\sqrt x $, $f_{N} (x)=(x+\sqrt x
+1)/3$, $f_{A} (x)=(x+1)/2$, $f_{R} (x)=2(x^{2}+x+1)/3(x+1)$, $f_{S}
(x)=\sqrt {(x^{2}+1)/2} $ and $f_{C} (x)=(x^{2}+1)/(x+1)$, $\forall x>0$,
$x\ne 1$. We have equality sign in (\ref{eq3}) iff $x=1$. For simplicity, let us
write
\begin{equation}
\label{eq4}
D_{AB} =b\,f_{AB} (a,b),
\end{equation}

\noindent where $f_{UV} (x)=f_{U} (x)-f_{V} (x)$, with $U\ge V$. Inequalities
appearing in (\ref{eq2}) admits 21 nonnegative differences. Some of these are
equal with multiplicative constants as given below:
\begin{align}
\label{eq5}
& \Delta :=3D_{CR} =2D_{AH} =2D_{CA} =D_{CH} =6D_{RA} =\textstyle{3 \over
2}D_{RH},\\
\label{eq6}
& h:=3D_{AN} =D_{AG} =\textstyle{3 \over 2}D_{NG}
\intertext{and}
\label{eq7}
& D_{CG} =3D_{RN}.
\end{align}

The measures $\Delta $ and $h$ appearing in (\ref{eq5}) and (\ref{eq6}) are respectively the \textit{triangular discrimination} \cite{lec} and \textit{Hellingar's distance} \cite{hel} respectively and are given by
\[
\Delta (a,b)=\frac{\left( {a-b} \right)^{2}}{a+b}
\]
and
\[
h(a,b)=\frac{\left( {\sqrt a -\sqrt b } \right)^{2}}{2}.
\]
More studied on these two measures can be seen in \cite{tan1, tan2, tan3}.

\bigskip
We shall improve considerably the inequalities given in (\ref{eq2}). For this we need first the convexity of the difference of means. In total, we have 21 differences. Some of them are equal to each other with some multiplicative constants.  Some of them are not convex and some of them are convex.

\section{Convexity of Difference of Means}

Let us prove now the convexity of some of the difference of means arising
due to inequalities (\ref{eq2}). In order to prove it we shall make use of the
following lemma \cite{tan1, tan2}.

\begin{lemma} Let $f:I\subset {\rm R}_{+} \to {\rm R}$ be a convex and
differentiable function satisfying $f(1)=0$. Consider a function
\[
\varphi_{f} (a,b)=af\left( {\frac{b}{a}} \right),
\quad
a,b>0,
\]
then the function $\varphi_{f} (a,b)$ is convex in ${\rm R}_{+}^{2} $.
Additionally, if $f^{\prime }(1)=0$, then the following inequality hold:
\[
0\le \varphi_{f} (a,b)\le \left( {\frac{b-a}{a}} \right)\varphi_{{f}'}
(a,b).
\]
\end{lemma}

In all the cases, it is easy to check that $f_{AB} (1)=f_{A} (1)-f_{B} (1)=1-1=0$. According to Lemma 2.1, it is sufficient to show the convexity of the functions $f_{AB} (x)$. It requires only to show that the second order
derivative of $f_{AB} (x)$ to be nonnegative for all $x>0$. Here below are the second order derivatives of the convex functions:
\begin{align}
& {f}''_{CS} (x)={f}''_{C} (x)-{f}''_{S} (x)=\frac{2\left[ {2\left( {2x^{2}+2}
\right)^{3/2}-\left( {x+1} \right)^{3}} \right]}{(x+1)^{3}\left( {2x^{2}+2}
\right)^{3/2}}, \notag\\
& {f}''_{CN} (x)={f}''_{C} (x)-{f}''_{N} (x)=\frac{48x^{3/2}+\left( {x+1}
\right)^{3}}{12x^{3/2}(x+1)^{3}}>0, \notag\\
& {f}''_{CG} (x)={f}''_{C} (x)-{f}''_{G}
(x)=\frac{16x^{3/2}+(x+1)^{3}}{4x^{3/2}(x+1)^{3}}>0, \notag\\
& {f}''_{SA} (x)={f}''_{S} (x)-{f}''_{A} (x)=\frac{1}{\left( {x^{2}+1}
\right)\sqrt {2x^{2}+2} }>0,\notag\\
& {f}''_{SN} (x)={f}''_{S} (x)-{f}''_{N} (x)=\frac{12x^{3/2}+\left( {x^{2}+1}
\right)\sqrt {2x^{2}+2} }{12x^{3/2}\left( {x^{2}+1} \right)\sqrt {2x^{2}+2}
}>0. \notag\\
& {f}''_{SG} (x)={f}''_{S} (x)-{f}''_{G} (x)=\frac{4x^{3/2}+\left( {x^{2}+1}
\right)\sqrt {2x^{2}+2} }{4x^{3/2}\left( {x^{2}+1} \right)\sqrt {2x^{2}+2}
}>0, \notag\\
& {f}''_{SH} (x)={f}''_{S} (x)-{f}''_{H} (x)=\frac{\left( {x+1}
\right)^{3}+4\left( {x^{2}+1} \right)\sqrt {2x^{2}+2} }{\left( {x+1}
\right)^{3}\left( {x^{2}+1} \right)\sqrt {2x^{2}+2} }>0, \notag\\
& {f}''_{AG} (x)={f}''_{A} (x)-{f}''_{G} (x)=\frac{1}{4x^{3/2}}>0, \notag\\
& {f}''_{AH} (x)={f}''_{A} (x)-{f}''_{H} (x)=\frac{4}{\left( {x+1}
\right)^{3}}>0 \notag
\intertext{and}
& {f}''_{RG} (x)={f}''_{R} (x)-{f}''_{G} (x)=\frac{16x^{3/2}+3\left( {x+1}
\right)^{3}}{4x^{3/2}\left( {x+1} \right)^{3}}>0. \notag
\end{align}

\noindent Since, $S\ge A$, this implies that $S^{3}\ge A^{3}$, i.e., $\left( {\sqrt
{\frac{x^{2}+1}{2}} } \right)^{3}-\left( {\frac{x+1}{2}} \right)^{3}\ge 0$.
This gives $2\left( {2x^{2}+2} \right)^{3/2}-\left( {x+1} \right)^{3}\ge 0$.
Thus we have ${f}''_{CS} (x)\ge 0$ for all $x>0$. The difference means ${D}_{SR}$, ${D}_{NH}$ and ${D}_{GH}$  are not convex.

\section{Inequalities among of Differences of Means}

In this section we shall bring sequence of inequalities based on the
differences arising due to (\ref{eq8}). This we shall present in two parts. The
results given in this section are based on the applications of the following
lemma \cite{tan1, tan2}:

\begin{lemma} Let $f_{1} ,f_{2} :I\subset {\rm R}_{+} \to {\rm R}$ be
two convex functions satisfying the assumptions:

(i) $f_{1} (1)=f_{1}^{\prime }(1)=0$, $f_{2} (1)=f_{2}^{\prime }(1)=0$;

(ii) $f_{1} $ and $f_{2} $ are twice differentiable in ${\rm R}_{+} $;

(iii) there exists the real constants $\alpha ,\beta $ such that $0\le
\alpha <\beta $ and
\[
\alpha \le \frac{f_{1}^{\prime \prime }(x)}{f_{2}^{\prime \prime }(x)}\le
\beta ,
\quad
f_{2}^{\prime \prime }(x)>0,
\]
for all $x>0$ then we have the inequalities:
\[
\alpha \mbox{\thinspace }\varphi_{f_{2} } (a,b)\le \varphi_{f_{1} }
(a,b)\le \beta \mbox{\thinspace }\varphi_{f_{2} } (a,b),
\]
for all $a,b\in (0,\infty )$, where the function $\phi_{(\cdot )} (a,b)$ is
as defined in Lemma 2.1.
\end{lemma}

The inequalities appearing in (\ref{eq2}) admits 21 nonnegative differences. The
differences satisfies some simple inequalities. These are given by the
following \textbf{pyramid}:

\[
D_{GH} ;
\]
\[
D_{NG} \le D_{NH} ;
\]
\[
D_{AN} \le D_{AG} \le D_{AH} ;
\]
\[
D_{RA} \le D_{RN} \le D_{RG} \le D_{RH} ;
\]
\[
D_{SR} \le D_{SA} \le D_{SN} \le D_{SG} \le D_{SH} ;
\]
\[
D_{CS} \le D_{CR} \le D_{CA} \le D_{CN} \le D_{CG} \le D_{CH},
\]

\bigskip
\noindent where the $D_{GH} =G-H$, $D_{NG} =N-G$, $D_{CS} =C-S$, etc. As we have seen
above some of these differences are equals with multiplicative constants.

\smallskip
The difference means $D_{SR} $, $D_{NH} $and $D_{GH} $ are not convex. The other convex measures satisfies some interesting inequalities with each other given by the theorem below.

\begin{theorem} The following inequalities hold:
\begin{equation}
\label{eq19}
D_{SA} \le \left\{ {\begin{array}{l}
 \textstyle{3 \over 4}D_{SN} \\
 \textstyle{1 \over 3}D_{SH} \le \textstyle{3 \over 4}D_{CR} \\
 \end{array}} \right\}\le \left\{ {\begin{array}{l}
 \textstyle{3 \over 7}D_{CN} \le \left\{ {\begin{array}{l}
 D_{CS} \\
 \textstyle{1 \over 3}D_{CG} \le \textstyle{3 \over 5}D_{RG} \\
 \end{array}} \right. \\
 \textstyle{1 \over 2}D_{SG} \le \textstyle{3 \over 5}D_{RG} \\
 \end{array}} \right\}\le 3D_{AN} .
\end{equation}
\end{theorem}

\begin{proof} We shall prove above result by parts. Here we shall use frequently the second order derivatives given in section 2.
\begin{enumerate}
\item \textbf{For }$\bf{D_{SA} \le \textstyle{3 \over 4}D_{SN}} $\textbf{: }Let us consider the function $_{\mathrm{\thinspace }}g_{SA\mathunderscore SN} (x)={{f}''_{SA} (x)} \mathord{\left/ {\vphantom {{{f}''_{SA} (x)} {{f}''_{SN} (x)}}} \right. \kern-\nulldelimiterspace} {{f}''_{SN} (x)}$. This gives
\[
{g}'_{SA\mathunderscore SN} (x)=-\frac{72\left( {x^{2}-1} \right)\sqrt x
\sqrt {2x^{2}+2} }{\left[ {24x^{3/2}+\left( {2x^{2}+2} \right)^{3/2}}
\right]^{2}}
\left\{ {{\begin{array}{*{20}c}
 {>0} & {x<1} \\
 {<0} & {x>1} \\
\end{array} }} \right.,
\]

Also we have
\begin{equation}
\label{eq20}
\beta_{SA\mathunderscore SN} =\mathop {\sup }\limits_{x\in (0,\infty )}
g_{SA\mathunderscore SN} (x)=g_{SA\mathunderscore SN} (1)=\frac{3}{4}.
\end{equation}

By the application Lemma 3.1 with (\ref{eq20}) we get the required result.

\item \textbf{For }$\bf{D_{SA} \le \textstyle{1 \over 3}D_{SH}} $\textbf{: }Let us consider the function $_{\mathrm{\thinspace }}g_{SA\mathunderscore SH} (x)={{f}''_{SA} (x)} \mathord{\left/ {\vphantom {{{f}''_{SA} (x)} {{f}''_{SH} (x)}}} \right. \kern-\nulldelimiterspace} {{f}''_{SH} (x)}$. This gives
\[
{g}'_{SA\mathunderscore SH} (x)=-\frac{12\left( {x-1} \right)\left( {x+1}
\right)^{2}\sqrt {2x^{2}+2} }{\left[ {\left( {x+1} \right)^{3}+4\left(
{x^{2}+1} \right)\sqrt {2x^{2}+2} } \right]^{2}}
\left\{ {{\begin{array}{*{20}c}
 {>0} & {x<1} \\
 {<0} & {x>1} \\
\end{array} }} \right.,
\]

Also we have
\begin{equation}
\label{eq21}
\beta_{SA\mathunderscore SH} =\mathop {\sup }\limits_{x\in (0,\infty )}
g_{SA\mathunderscore SH} (x)=g_{SA\mathunderscore SH} (1)=\frac{1}{3}.
\end{equation}

By the application Lemma 3.1 with (\ref{eq21}) we get the required result.

\item \textbf{For }$\bf{D_{SH} \le \textstyle{9 \over 4}D_{CR}} $\textbf{: }Let us consider the function $_{\mathrm{\thinspace }}g_{SH\mathunderscore CR} (x)={{f}''_{SH} (x)} \mathord{\left/ {\vphantom {{{f}''_{SH} (x)} {{f}''_{CR} (x)}}} \right. \kern-\nulldelimiterspace} {{f}''_{CR} (x)}$. This gives
\[
{g}'_{SH\mathunderscore CR} (x)=-\frac{9\left( {x-1} \right)\left( {x+1}
\right)^{2}}{8\left( {x^{2}+1} \right)^{2}\sqrt {2x^{2}+2} }
\left\{ {{\begin{array}{*{20}c}
 {>0} & {x<1} \\
 {<0} & {x>1} \\
\end{array} }} \right..
\]

Also we have
\begin{equation}
\label{eq22}
\beta_{SA\mathunderscore SH} =\mathop {\sup }\limits_{x\in (0,\infty )}
g_{SA\mathunderscore SH} (x)=g_{SA\mathunderscore SH} (1)=\frac{1}{3}.
\end{equation}

By the application Lemma 3.1 with (\ref{eq23}) we get the required result.

\item \textbf{For }$\bf{D_{CR} \le \textstyle{4 \over 7}D_{CN}} $\textbf{: }Let us consider the function $_{\mathrm{\thinspace }}g_{CR\mathunderscore CN} (x)={{f}''_{CR} (x)} \mathord{\left/ {\vphantom {{{f}''_{CR} (x)} {{f}''_{CN} (x)}}} \right. \kern-\nulldelimiterspace} {{f}''_{CN} (x)}$. This gives
\[
{g}'_{CR\mathunderscore CN} (x)=-\frac{48\sqrt x \left( {x-1} \right)\left(
{x+1} \right)^{2}}{\left[ {48x^{3/2}+\left( {x+1} \right)^{3}} \right]^{2}}
\left\{ {{\begin{array}{*{20}c}
 {>0} & {x<1} \\
 {<0} & {x>1} \\
\end{array} }} \right..
\]

Also we have
\begin{equation}
\label{eq23}
\beta_{CR\mathunderscore CN} =\mathop {\sup }\limits_{x\in (0,\infty )}
g_{CR\mathunderscore CN} (x)=g_{CR\mathunderscore CN} (1)=\frac{4}{7}.
\end{equation}

By the application Lemma 3.1 with (\ref{eq22}) we get the required result.

\item \textbf{For }$\bf{D_{CR} \le \textstyle{2 \over 3}D_{SG}} $\textbf{: }Let us consider the function $_{\mathrm{\thinspace }}g_{CR\mathunderscore SG} (x)={{f}''_{CR} (x)} \mathord{\left/ {\vphantom {{{f}''_{CR} (x)} {{f}''_{SG} (x)}}} \right. \kern-\nulldelimiterspace} {{f}''_{SG} (x)}$. This gives
\[
{g}'_{CR\mathunderscore SG} (x)=-\frac{16\left( {x-1} \right)\sqrt x \sqrt
{2x^{2}+2} \times v_{1} (x)}{\left( {x+1} \right)^{3}\left[ {4x^{3/2}+\left(
{x^{2}+1} \right)\sqrt {2x^{2}+2} } \right]^{2}}
\left\{ {{\begin{array}{*{20}c}
 {>0} & {x<1} \\
 {<0} & {x>1} \\
\end{array} }} \right.,
\]
where
\[
v_{1} (x)=\left( {x^{2}+1} \right)^{2}\sqrt {2x^{2}+2} -8x^{5/2}
=8\left[ {\left( {\sqrt {\frac{x^{2}+1}{2}} } \right)^{5}-\left( {\sqrt x }
\right)^{5}} \right]>0_{\mathrm{.}}
\]

Above expression holds since $S>G$, $\forall x>0,\,x\ne 1$. Also we have
\begin{equation}
\label{eq24}
\beta_{CR\mathunderscore SG} =\mathop {\sup }\limits_{x\in (0,\infty )}
g_{CR\mathunderscore SG} (x)=g_{CR\mathunderscore SG} (1)=\frac{2}{3}.
\end{equation}

By the application Lemma 3.1 with (\ref{eq24}) we get the required result.

\item \textbf{For }$\bf{D_{SN} \le \textstyle{4 \over 7}D_{CN}} $\textbf{: } We have $\beta_{SN\mathunderscore CN} ={{f}''_{SN} (1)} \mathord{\left/ {\vphantom {{{f}''_{SN} (1)} {{f}''_{CN} (1)}}} \right. \kern-\nulldelimiterspace} {{f}''_{CN} (1)}=\textstyle{4 \over 7}$. Now, we have to show that $\textstyle{4 \over 7}D_{CN} -D_{SN} \ge 0$, i.e., $\textstyle{1 \over 7}\left( {4C+3N-7S} \right)\ge 0$. We can write $\textstyle{4 \over 7}D_{CN} -D_{SN}{\mathrm{\thinspace }}=b\,f_{CN\mathunderscore SN} (a/b)$, where
\[
f_{CN\mathunderscore SN} (x)=\textstyle{4 \over 7}f_{SN} (x)-f_{CN} (x)
=\frac{1}{14\left( {x+1} \right)}\times v_{2} (x),
\]
where
\[
v_{2} (x)=10x^{2}+10+4x+2x^{3/2}+2\sqrt x -7\left( {x+1} \right)\sqrt
{2x^{2}+2} .
\]

In order to prove the non-negativity of $v_{2} (x)$, let us consider the
function
\begin{align}
h_{2} (x)& =\left( {10x^{2}+10+4x+2x^{3/2}+2\sqrt x } \right)^{2}-\left(
{7\left( {x+1} \right)\sqrt {2x^{2}+2} } \right)^{2}\notag\\
& =\left( {2x^{2}+48x^{3/2}+68x+48\sqrt x +2} \right)\left( {\sqrt x -1}
\right)^{4}.\notag
\end{align}

Since $h_{2} (x)\ge 0$, giving $v_{2} (x)\ge 0$, $\forall x>0$. This implies
that $f_{CN\mathunderscore SN} (x)\ge 0$, $\forall x>0$, hence proving the
required result.

\textbf{Argument:} \textit{Let }$a$\textit{ and }$b$\textit{ two positive numbers, i.e., }$a>0$\textit{ and }$b>0$\textit{. If }$a^{2}-b^{2}\ge 0$\textit{, then we can conclude that }$a\ge b$\textit{ because }$a-b=({a^{2}-b^{2})}
\mathord{\left/ {\vphantom {{a^{2}-b^{2})} {(a+b)}}} \right.
\kern-\nulldelimiterspace} {(a+b)}$\textit{. We have used this argument to prove }$v_{2} (x)\ge 0, \forall x>0. $

\item \textbf{For }$\bf{D_{SN} \le \textstyle{2 \over 3}D_{SG}} $\textbf{: }Let us consider the function $_{\mathrm{\thinspace }}g_{SN\mathunderscore SG} (x)={{f}''_{SN} (x)} \mathord{\left/ {\vphantom {{{f}''_{SN} (x)} {{f}''_{SG} (x)}}} \right. \kern-\nulldelimiterspace} {{f}''_{SG} (x)}$. This gives
\[
{g}'_{SN\mathunderscore SG} (x)=-\frac{4\sqrt x \left( {x^{2}-1}
\right)\sqrt {2x^{2}+2} }{\left[ {4x^{3/2}+\left( {x^{2}+1} \right)\sqrt
{2x^{2}+2} } \right]^{2}}
\left\{ {{\begin{array}{*{20}c}
 {>0} & {x<1} \\
 {<0} & {x>1} \\
\end{array} }} \right..
\]

Also we have
\begin{equation}
\label{eq25}
\beta_{SBN\mathunderscore SG} =\mathop {\sup }\limits_{x\in (0,\infty )}
g_{SN\mathunderscore SG} (x)=g_{SN\mathunderscore SG} (1)=\frac{2}{3}.
\end{equation}

By the application Lemma 3.1 with (\ref{eq25}) we get the required result.

\item \textbf{For }$\bf{D_{CN} \le \textstyle{7 \over 3}D_{CS}} $\textbf{: }We have $\beta_{CN\mathunderscore CS} ={{f}''_{CN} (1)} \mathord{\left/ {\vphantom {{{f}''_{CN} (1)} {{f}''_{CS} (1)}}} \right. \kern-\nulldelimiterspace} {{f}''_{CS} (1)}=\textstyle{7 \over 3}$. Now, we have to show that $\textstyle{7 \over 3}D_{CS} -D_{CN} \ge 0$, i.e., $\textstyle{1 \over 3}\left( {4C+3N-7S} \right)\ge 0$. This is true in view of part 6.

\item \textbf{For }$\bf{D_{CS} \le 3D_{AN}} $\textbf{: }Let us consider the function $_{\mathrm{\thinspace }}g_{CS\mathunderscore AN} (x)={{f}''_{CS} (x)} \mathord{\left/ {\vphantom {{{f}''_{CS} (x)} {{f}''_{AN} (x)}}} \right. \kern-\nulldelimiterspace} {{f}''_{AN} (x)}$. This gives
\[
{g}'_{CS\mathunderscore AN} (x)=-\frac{18\sqrt x \left( {x-1} \right)\times
v_{3} (x)}{\left( {x^{2}+1} \right)^{2}\left( {x+1} \right)^{2}\sqrt
{2x^{2}+2} }
\left\{ {{\begin{array}{*{20}c}
 {>0} & {x<1} \\
 {<0} & {x>1} \\
\end{array} }} \right.,
\]
where
\begin{align}
v_{3} (x)& =4\left( {x^{2}+1} \right)^{2}\sqrt {2x^{2}+2} -\left( {x+1}
\right)^{5}\notag\\
& =32\left[ {\left( {\sqrt {\frac{x^{2}+1}{2}} } \right)^{2}-\left(
{\frac{x+1}{2}} \right)^{5}} \right]>0.\notag
\end{align}

Above expression holds since $S>A$, $\forall x>0,\,x\ne 1$. Also we have
\begin{equation}
\label{eq26}
\beta_{CS\mathunderscore AN} =\mathop {\sup }\limits_{x\in (0,\infty )}
g_{CS\mathunderscore AN} (x)=g_{CS\mathunderscore AN} (1)=3.
\end{equation}

By the application Lemma 3.1 with (\ref{eq26}) we get the required result.

\item \textbf{For }$\bf{D_{CN} \le \textstyle{7 \over 9}D_{CG} }$\textbf{: }Let us consider the function $_{\mathrm{\thinspace }}g_{CN\mathunderscore CG} (x)={{f}''_{CN} (x)} \mathord{\left/ {\vphantom {{{f}''_{CN} (x)} {{f}''_{CG} (x)}}} \right. \kern-\nulldelimiterspace} {{f}''_{CG} (x)}$. This gives
\[
{g}'_{CN\mathunderscore CG} (x)=-\frac{16\sqrt x \left( {x-1} \right)\left(
{x+1} \right)^{2}}{\left[ {16x^{3/2}+\left( {x+1} \right)^{2}} \right]^{2}}
\left\{ {{\begin{array}{*{20}c}
 {>0} & {x<1} \\
 {<0} & {x>1} \\
\end{array} }} \right..
\]

Also we have
\begin{equation}
\label{eq27}
\beta_{CN\mathunderscore CG} =\mathop {\sup }\limits_{x\in (0,\infty )}
g_{CN\mathunderscore CG} (x)=g_{CN\mathunderscore CG} (1)=\frac{7}{9}.
\end{equation}

By the application Lemma 3.1 with (\ref{eq27}) we get the required result.

\item \textbf{For }$\bf{D_{SG} \le \textstyle{6 \over 5}D_{RG}} $\textbf{: }We have $\beta_{SG\mathunderscore RG} ={{f}''_{SG} (1)} \mathord{\left/ {\vphantom {{{f}''_{SG} (1)} {{f}''_{RG} (1)}}} \right. \kern-\nulldelimiterspace} {{f}''_{RG} (1)}=\textstyle{6 \over 5}$. Now, we have to show that $\textstyle{6 \over 5}D_{RG} -D_{SG} \ge 0$, i.e., $\textstyle{1 \over 5}\left( {6R-G-5S} \right)\ge 0$. We can write $\textstyle{6 \over 5}D_{RG} -D_{SG}{\mathrm{\thinspace }}=b\,f_{SG\mathunderscore RG} (a/b)$, where
\[
f_{SG\mathunderscore RG} (x)=\textstyle{6 \over 5}f_{RG} (x)-f_{SG} (x)
=\frac{1}{10\left( {x+1} \right)}\times v_{3} (x),
\]
where
\[
v_{3} (x)=8\left( {x^{2}+x+1} \right)-2\sqrt x \left( {x+1} \right)-5\left(
{x+1} \right)\sqrt {2x^{2}+2} .
\]

In order to prove the non-negativity of $v_{3} (x)$, let us consider the
function
\begin{align}
h_{3} (x)& =\left[ {8\left( {x^{2}+x+1} \right)-2\sqrt x \left( {x+1} \right)}
\right]^{2}-\left( {5\left( {x+1} \right)\sqrt {2x^{2}+2} } \right)^{2}\notag\\
& =\left( {14x^{2}+24x^{3/2}+44x+24\sqrt x +14} \right)\left( {\sqrt x -1}
\right)^{4}.\notag
\end{align}

Since $h_{3} (x)\ge 0$, giving $v_{3} (x)\ge 0$, $\forall x>0$. The
non-negativity of the expression $8\left( {x^{2}+x+1} \right)-2\sqrt x
\left( {x+1} \right)$ can be shown easily following the same lines, i.e.
\begin{align}
& \left[ {4\left( {x^{2}+x+1} \right)} \right]^{2}-\left[ {\sqrt x \left(
{x+1} \right)} \right]^{2}\notag\\
& \hspace{20pt} =16x^{4}+31x^{3}+46x^{2}+31x+16>0.\notag
\end{align}

This implies that $f_{SG\mathunderscore RG} (x)\ge 0$, $\forall x>0$, hence
proving the required result.

\item \textbf{For }$D_{CG} \le \textstyle{9 \over 5}D_{RG} $\textbf{: }Let us consider the function $_{\mathrm{\thinspace }}g_{CG\mathunderscore RG} (x)={{f}''_{CG} (x)} \mathord{\left/ {\vphantom {{{f}''_{CG} (x)} {{f}''_{RG} (x)}}} \right. \kern-\nulldelimiterspace} {{f}''_{RG} (x)}$. This gives
\[
{g}'_{CG\mathunderscore RG} (x)=-\frac{144\sqrt x \left( {x-1} \right)\left(
{x+1} \right)^{2}}{\left[ {16x^{3/2}+3\left( {x+1} \right)^{3}} \right]^{2}}
\left\{ {{\begin{array}{*{20}c}
 {>0} & {x<1} \\
 {<0} & {x>1} \\
\end{array} }} \right..
\]

Also we have
\begin{equation}
\label{eq28}
\beta_{CG\mathunderscore RG} =\mathop {\sup }\limits_{x\in (0,\infty )}
g_{CG\mathunderscore RG} (x)=g_{CG\mathunderscore RG} (1)=\frac{9}{5}.
\end{equation}

By the application Lemma 3.1 with (\ref{eq28}) we get the required result.

\item \textbf{For }$D_{RG} \le 5D_{AN} $\textbf{: }Let us consider the function $_{\mathrm{\thinspace }}g_{RG\mathunderscore AN} (x)={{f}''_{RG} (x)} \mathord{\left/ {\vphantom {{{f}''_{RG} (x)} {{f}''_{AN} (x)}}} \right. \kern-\nulldelimiterspace} {{f}''_{AN} (x)}$. This gives
\[
{g}'_{RG\mathunderscore AN} (x)=-\frac{24\sqrt x \left( {x-1}
\right)}{\left( {x+1} \right)^{4}}
\left\{ {{\begin{array}{*{20}c}
 {>0} & {x<1} \\
 {<0} & {x>1} \\
\end{array} }} \right..
\]

Also we have
\begin{equation}
\label{eq29}
\beta_{RG\mathunderscore AN} =\mathop{\sup }\limits_{x\in (0,\infty )}
g_{RG\mathunderscore AN} (x)=g_{RG\mathunderscore AN} (1)=5.
\end{equation}

By the application Lemma 3.1 with (\ref{eq29}) we get the required result.
\end{enumerate}
\end{proof}

\begin{remark} The above 13 parts allows writing inequalities in their
equivalent forms:
\begin{multicols}{2}
\begin{enumerate}
\item $\frac{2G+S}{3}\le N;$
\item $\frac{2C+7G}{9}\le N;$
\item $\frac{S+3N}{4}\le A;$
\item $\frac{2S+H}{3}\le A;$
\item $\frac{2C+3N}{7}\le R;$
\item $\frac{G+5S}{6}\le R$
\item $\frac{4G+5C}{9}\le R;$
\item $S\le \frac{4C+3N}{7};$
\item $9R+4S\le 9C+4H;$
\item $2G+3C\le 2S+3R;$
\item $3N+C\le 3A+S;$
\item $5N+R\le 5A+G.$
\end{enumerate}
\end{multicols}
\end{remark}

Based on these equivalent versions, here below is an improvement over
inequalities appearing in (\ref{eq2}):

\begin{proposition} The following inequalities hold:
\[
H\le G\le \left\{ {\begin{array}{l}
 \textstyle{{2G+S} \over 3} \\
 \textstyle{{2C+7G} \over 9} \\
 \end{array}} \right\}\le N\le \textstyle{{3A+S-C} \over 3}\le \left\{
{\begin{array}{l}
 \textstyle{{S+3N} \over 4} \\
 \textstyle{{2S+H} \over 3} \\
 \end{array}} \right\}\le A\le \left\{ {\begin{array}{l}
 \textstyle{{G+5S} \over 6} \\
 \textstyle{{4G+5C} \over 9} \\
 \end{array}} \right\}\le
\]
\begin{equation}
\label{eq30}
\le R\le \left\{ {\begin{array}{l}
 S \\
 5A+G-5N \\
 \end{array}} \right\}\le \textstyle{{9C+4H-9R} \over 4}\le C\le
\textstyle{{2S+3R-2G} \over 3}.
\end{equation}
\end{proposition}

Inequalities appearing (\ref{eq30}) can be proved by using similar arguments of Theorem 3.1.

\subsection{Proportionality Relations among Means}

As a part of (\ref{eq19}), let us consider the following inequalities:
\begin{equation}
\label{eq31}
\textstyle{1 \over 4}\Delta \le \textstyle{3 \over 7}D_{CN} \le \textstyle{1
\over 3}D_{CG} \le \textstyle{3 \over 5}D_{RG} \le h.
\end{equation}
The expression (\ref{eq31}) has six means istead of seven. For simplicity, let us
rewrite the expression (\ref{eq31}):
\begin{equation}
\label{eq32}
W_{1} \le W_{2} \le W_{3} \le W_{4} \le W_{5} ,
\end{equation}

\bigskip
\noindent where for example $W_{1} =\textstyle{1 \over 4}\Delta $, $W_{2}
=\textstyle{3 \over 7}D_{CN} $, $W_{9} =\textstyle{1 \over 3}D_{CG} $, etc.
The inequalities (\ref{eq32}) again admits 10 nonnegative differences. These
differences satisfies some natural inequalities given in a \textbf{pyramid}
below:

\[
D_{W_{2} W_{1} }^{1} ,
\]
\[
D_{W_{3} W_{2} }^{2} \le D_{W_{3} W_{1} }^{3} ,
\]
\[
D_{W_{4} W_{3} }^{4} \le D_{W_{4} W_{2} }^{5} \le D_{W_{4} W_{1} }^{6} ,
\]
\[
D_{W_{5} W_{4} }^{7} \le D_{W_{5} W_{3} }^{8} \le D_{W_{5} W_{2} }^{9} \le
D_{W_{5} W_{1} }^{10} ,
\]

\bigskip
\noindent where $D_{W_{2} W_{1} }^{1} :=W_{2} -W_{1} $, $D_{W_{7} W_{6} }^{16} :=W_{7}
-W_{6} $, etc. Interestingly, the above 10 nonnegative differences are
equals to each other by some multiplicative constants:
\begin{align}
& \textstyle{7 \over 2}D_{W_{2} W_{1} }^{1} =\textstyle{{21} \over 8}D_{W_{3}
W_{2} }^{2} =\textstyle{3 \over 2}D_{W_{3} W_{1} }^{3} =\textstyle{{15}
\over 8}D_{W_{4} W_{3} }^{4} =\textstyle{{35} \over {32}}D_{W_{4} W_{2}
}^{5} =\notag\\
\label{eq33}
& \hspace{20pt} =\textstyle{5 \over 6}D_{W_{4} W_{1} }^{6} =\textstyle{5 \over 4}D_{W_{5}
W_{4} }^{7} =\textstyle{3 \over 4}D_{W_{5} W_{3} }^{8} =\textstyle{7 \over
{12}}D_{W_{5} W_{2} }^{9} =\textstyle{1 \over 2}D_{W_{5} W_{1} }^{10}
=\frac{\left( {\sqrt a -\sqrt b } \right)^{4}}{a+b}.
\end{align}

Based on the expressions (\ref{eq4}), (\ref{eq5}), (\ref{eq6})  and (\ref{eq33}) we have the following proportionality relations among the six means:
\begin{multicols}{2}
\begin{enumerate}
\item $4A=2(C+H)=3R+H;$
\item $3R=C+2A=2C+H;$
\item $3N=2A+G;$
\item $3C+2H=3R+2A;$
\item $C+6A=H+6R;$
\item $C+3N=G+3R;$
\item $3N+2A=2C+2H+G;$
\item $27R+2G=14A+9C+6N;$
\item $3\left( {N+3R} \right)=8A+3C+G;$
\item $3G+8H+9C=3R+8A+9N;$
\item $4G+14H+17C=9R+14A+12N;$
\item $5G+24H+31C=21R+24A+15N.$
\end{enumerate}
\end{multicols}

\end{document}